\documentclass[11pt]{amsart}
\usepackage{palatino}
\usepackage{amsfonts}
\usepackage{amssymb}
\usepackage{amscd}
\usepackage{xypic}
\usepackage[breaklinks,bookmarksopen,bookmarksnumbered]{hyperref}
\usepackage[dvips]{graphicx}

\newcommand{\dblq}{{/\!/}}
\newtheorem{theorem}{Theorem}[section]
\newtheorem{proposition}[theorem]{Proposition}

\newtheorem{lemma}[theorem]{Lemma}
\theoremstyle{definition}
\newtheorem{definition}[theorem]{Definition}
\newtheorem{remark}[theorem]{Remark}

\newtheorem{conjecture/question}[theorem]{Conjecture/Question}

\newtheorem{remark/definition}[theorem]{Remark/Definition}
\newtheorem{terminology/notation}[theorem]{Terminology/Notation}

\def\PP{{\textbf P}}
\def\OO{\mathcal{O}}

\def\cA{\mathcal{A}}
\def\F{\mathcal{F}}
\def\P{\mathcal{P}}

\def\E{\mathcal{E}}

\def\I{\mathcal{I}}

\def\cM{\mathcal{M}}
\def\cR{\mathcal{R}}
\def\ra{\mathfrak{Rat}}
\def\rr{\overline{\mathcal{R}}}

\def\Pic0{{\rm Pic}^0(X)}

\def\mm{\overline{\mathcal{M}}}

\begin{document}
\title{Prym varieties and moduli of polarized Nikulin surfaces}

\author[G. Farkas]{Gavril Farkas}

\address{Humboldt-Universit\"at zu Berlin, Institut f\"ur Mathematik,  Unter den Linden 6
\hfill \newline\texttt{}
 \indent 10099 Berlin, Germany} \email{{\tt farkas@math.hu-berlin.de}}
\thanks{}

\author[A. Verra]{Alessandro Verra}
\address{Universit\'a Roma Tre, Dipartimento di Matematica, Largo San Leonardo Murialdo \hfill
 \newline \indent 1-00146 Roma, Italy}
 \email{{\tt
verra@mat.uniroma3.it}}

\begin{abstract} We present a structure theorem for the moduli space $\cR_7$ of Prym curves of genus $7$ as a projective bundle over the moduli space of $7$-nodal rational curves. The existence of this parametrization implies the unirationality of $\cR_7$ and that of the moduli space of Nikulin surfaces of genus $7$, as well as the rationality of the moduli space of Nikulin surfaces of genus $7$ with a distinguished line. Using the results in genus $7$, we then establish that $\cR_8$ is uniruled.
\end{abstract}

\maketitle
\vskip 6pt

\section{Introduction}
A polarized Nikulin surface of genus $g$ is a smooth polarized $K3$ surface $(S, \mathfrak{c})$, where $\mathfrak{c}\in \mbox{Pic}(S)$ with $\mathfrak{c}^2=2g-2$, equipped with a double cover
$f:\widetilde{S}\rightarrow S$ branched along disjoint rational curves $N_1, \ldots, N_8\subset S$, such that $\mathfrak{c}\cdot N_i=0$ for $i=1, \ldots, 8$. Denoting by $e\in \mbox{Pic}(S)$ the class defined by the equality $e^{\otimes 2}=\OO_S(\sum_{i=1}^8 N_i)$, one forms the \emph{Nikulin lattice}
$$\mathfrak{N}:=\Bigl\langle \OO_S(N_1), \ldots, \OO_S(N_8), e\Bigr\rangle$$
and obtains a primitive embedding $j:\Lambda_g:=\mathbb Z\cdot [\mathfrak{c}]\oplus \mathfrak{N}\hookrightarrow \mbox{Pic}(S)$. Nikulin surfaces of genus $g$ form an irreducible $11$-dimensional moduli space $\F_g^{\mathfrak{N}}$ which has been  studied  from a lattice-theoretic point of view in  \cite{Do1} and \cite{vGS}.
The connection between $\F_g^{\mathfrak{N}}$ and the moduli space $\cR_g$ of pairs $[C, \eta]$, where $C$ is a curve of genus $g$ and $\eta\in \mbox{Pic}^0(C)[2]$ is a non-trivial $2$-torsion point, has been pointed out in \cite{FV} and used to describe $\cR_g$ in small genus. Over $\F_g^{\mathfrak{N}}$ one considers the open set in a tautological $\PP^g$-bundle
$$\mathcal{P}_g^{\mathfrak{N}}:=\Bigl\{\bigl[S, j:\Lambda_g\hookrightarrow \mbox{Pic}(S), C\bigr]: C\in |\mathfrak{c}| \mbox{ is a smooth curve of genus  } g\Bigr\},$$
which is endowed with the two projection maps
$$\xymatrix{
  & \mathcal{P}_g^{\mathfrak{N}} \ar[dl]_{p_g} \ar[dr]^{\chi_g} & \\
   \F_g^{\mathfrak{N}} & & \cR_{g}       \\
                 }$$
defined by $p_g([S, j, C]):=[S,j]$ and $\chi_g([S, j, C]):=[C, e_C:=e\otimes \OO_C]$ respectively.

\vskip 4pt

Observe that $\mbox{dim}(\mathcal{P}_7^{\mathfrak{N}})=\mbox{dim}(\cR_7)=18$. The map $\chi_7:\mathcal{P}_7^{\mathfrak{N}}\dashrightarrow \cR_7$ is a birational isomorphism, precisely  $\cR_7$ is birational to a Zariski locally trivial $\PP^7$-bundle over $\F_7^{\mathfrak{N}}$. This is reminiscent of Mukai's well-known result \cite{Mu}:  The moduli space $\cM_{11}$ of curves of genus $11$ is birational to a projective bundle over the moduli space $\F_{11}$ of polarized $K3$ surfaces of genus $11$. Note that $\cM_{11}$ and $\cR_7$ are the only known examples of moduli spaces of curves admitting a non-trivial fibre bundle structure over a moduli space of polarized $K3$ surfaces. Here we describe the structure of $\F_7^{\mathfrak{N}}$:
\begin{theorem}\label{unir7}
The Nikulin moduli space $\F_7^{\mathfrak{N}}$ is unirational. The Prym moduli space $\cR_7$ is birationally isomorphic to a $\PP^7$-bundle over $\F_7^{\mathfrak{N}}$. It follows that $\cR_7$ is unirational as well.
\end{theorem}
It is well-known that $\cR_g$ is unirational for $g\leq 6$, see \cite{Do}, \cite{ILS}, \cite{V}, and even rational for $g\leq 4$, see \cite{Do2}, \cite{Cat}. On the other hand, the Deligne-Mumford moduli space $\rr_g$ of stable Prym curves of genus $g$ is a variety of general type for $g\geq 14$, whereas $\mbox{kod}(\rr_{12})\geq 0$, see \cite{FL} for the cases $g\neq 15$ and \cite{Br} for the case $g=15$. Nothing seems to be known about the Kodaira dimension of $\rr_g$, for $g=9, 10, 11$.

\vskip 3pt

We now discuss the structure of $\F_7^{\mathfrak{N}}$. For each positive $g$, we denote by $$\ra_g:=\mm_{0,2g}/\mathbb Z_2^{\oplus g}\rtimes \mathfrak{S}_g$$ the moduli space of $g$-nodal stable rational curves. The action of the group $\mathbb Z_2^{\oplus g}$ is given by permuting the marked points labeled by $\{1, 2\}, \ldots, \{2g-1,2g\}$ respectively, while the symmetric group $\mathfrak{S}_g$ acts by permuting the $2$-cycles $(1,2), \ldots, (2g-1,2g)$ respectively. The variety $\ra_g$, viewed as a subvariety of $\mm_g$, has been studied by Castelnuovo \cite{Cas} at the end of the 19th century in the course of his famous attempt to prove the Brill-Noether Theorem, as well as much more recently, for instance  in \cite{GKM} \footnote{Unfortunately, in \cite{GKM} the notation $\rr_g$ (reserved for the Prym moduli space) is proposed for what we denote in this paper by $\ra_g$.}, in the context of determining the ample cone of $\mm_g$.
Using the identification $\mbox{Sym}^2(\PP^1)\cong \PP^2$, we obtain a birational isomorphism
$$\ra_g\cong \mbox{Hilb}^g(\PP^2) \dblq PGL(2),$$
where $PGL(2)\subset PGL(3)$ is regarded as the group of projective automorphisms of $\PP^2$ preserving the image of a fixed smooth conic in $\PP^2$.

\vskip 5pt

Let us fix once and for all a smooth rational quintic curve $R\subset \PP^5$. For  general points $x_1, y_1, \ldots, x_{7}, y_7\in R$, we note that $\bigl[R, (x_1+y_1)+ \cdots +(x_{7}+y_{7})\bigr]\in \ra_7$. We denote by
$$N_1:=\langle x_1, y_1\rangle, \ldots, N_7:=\langle x_{7}, y_{7}\rangle \in G(2, 6),$$
the corresponding bisecant lines to $R$ and observe that $C:=R\cup N_1\cup \ldots \cup N_7$ is a nodal curve of genus $7$ and degree $12$ in $\PP^5$.
By writing down the Mayer-Vietoris sequence
for $C$, we find the following identifications:
$$H^0(C, \OO_C(1))\cong H^0(\OO_R(1)) \ \mbox{ and } \ H^0(C, \OO_C(2))\cong H^0(\OO_R(2))\oplus \Bigl(\oplus_{i=1}^7 H^0(\OO_{N_i})\Bigr).$$
It can easily be checked that the base locus
$$S:=\mathrm{Bs}\ \bigl|\I_{C/\PP^5}(2)\bigr|$$ is a
smooth $K3$ surface which is a complete intersection of three quadrics in $\PP^5$. Obviously, $S$ is equipped with the seven lines $N_1, \ldots, N_7$. In fact, $S$ carries an eight line as well! If $H\in |\OO_S(1)|$ is a hyperplane section, after setting
$$N_8:=2R+N_1+\cdots+N_7-2H\in \mbox{Div}(S),$$
we compute that $N_8^2=-2, N_8\cdot H=1$ and $N_8\cdot N_i=0$, for $i=1, \ldots, 7$. Therefore $N_8$ is equivalent to an effective divisor on $S$, which is embedded in $\PP^5$ as a line by the linear system $|\OO_S(1)|$. Furthermore,
$$N_1+\cdots+N_8=2(R+N_1+\cdots+N_7-H)\in \mathrm{Pic}(S),$$
hence by denoting $e:=R+N_1+\cdots+N_7-H$, we obtain an embedding $\mathfrak{N}\hookrightarrow \mbox{Pic}(S)$. Moreover $C\cdot N_i=0$ for $i=1, \ldots, 8$ and we may view $\Lambda_7\hookrightarrow \mbox{Pic}(S)$. In this way $S$ becomes a Nikulin surface of genus $7$.

\vskip 5pt

We introduce the moduli space $\widehat{\F}_g^{\mathfrak{N}}$ of \emph{decorated} Nikulin surfaces consisting of polarized Nikulin surfaces $\bigl[S, j:\Lambda_g\hookrightarrow \mbox{Pic}(S)\bigr]$
of genus $g$, together with a distinguished line $N_8\subset S$ viewed as a component of the branch divisor of the double covering $f:\widetilde{S}\rightarrow S$. There is an obvious forgetful map $\widehat{\F}_g^{\mathfrak{N}}\rightarrow \F_g^{\mathfrak{N}}$ of degree $8$. Having specified $N_8\subset S$, we can also specify the divisor $N_1+\cdots+N_7\subset S$ such that $e^{\otimes 2}=\OO_S(N_1+\cdots+N_7+N_8)$. We summarize what has been discussed so far and refer to Section 2 for further details:

\begin{theorem}\label{m14}
The rational map $\varphi:\ra_7\dashrightarrow \widehat{\F}_7^{\mathfrak{N}}$ given by
$$\varphi\Bigl(\bigl[R, (x_1+y_1)+ \cdots +(x_7+y_7)\bigr]\Bigr):=\Bigl[S, \OO_S (R+N_1+\cdots+N_7), N_8\Bigr]$$
is a birational isomorphism.
\end{theorem}
A construction of the inverse map $\varphi^{-1}$ using the geometry of Prym canonical curves of genus $7$ is presented in Section 2.
The moduli space $\ra_g$ is related to the configuration space
$$U_g^2:=\mbox{Hilb}^g(\PP^2)\dblq PGL(3)$$ of $g$ unordered points in the plane. Using the isomorphism $PGL(3)/PGL(2)\cong \PP^5$, we observe in Section 2 that there exists a
(locally trivial) $\PP^5$-bundle structure $\ra_g\dashrightarrow U^2_g$. In particular $\ra_g$ is rational whenever $U^2_g$ is. Since the rationality of $U^2_7$ has been established by Katsylo \cite{Ka} (see also \cite{Bo}), we are led to the following result:

\begin{theorem}\label{rat8}
The moduli space $\widehat{\F}_7^{\mathfrak N}$ of decorated Nikulin surfaces of genus $7$ is rational.
\end{theorem}

Putting together Theorems \ref{m14} and \ref{rat8}, we conclude that  there exists a dominant rational map $\PP^{18}\dashrightarrow \cR_7$ of degree $8$. We are not aware of any dominant map from a rational variety to $\cR_7$ of degree smaller than $8$. It would be very interesting to know whether $\cR_7$ itself is a rational variety. We recall that although $\cM_g$ is known to be rational for $g\leq 6$ (see \cite{Bo} and the references therein), the rationality of $\cM_7$ is an open problem.

\vskip 3pt

We sum up the construction described above in the following commutative diagram:
$$
     \xymatrix{
         \mm_{0,14} \ar[r]^{(2^7\cdot 7!):1} \ar@{-->}[d]_{} & \ra_{7} \ar@{-->}[d]_{\cong} &  \\
          \F_7^{\mathfrak{N}}        & \widehat{\F}_7^{\mathfrak{N}} \ar[l]^{8:1}  \ar[r]^{\PP^5} & U^2_7\\}
$$

The concrete geometry of $\cR_7$ by means of polarized Nikulin surfaces has direct consequences concerning the Kodaira dimension of $\rr_8$. The projective bundle structure of $\cR_7$ over $\F_7^{\mathfrak{N}}$ can be lifted to a boundary divisor of $\rr_8$. Denoting by $\pi:\rr_g\rightarrow \mm_g$ the map forgetting the Prym structure, one has the  formula
$$
\pi^*(\delta_0)=\delta_0^{'}+\delta_0^{''}+2\delta_{0}^{\mathrm{ram}}\in CH^1(\rr_g),
$$
where $\delta_0^{'}:=[\Delta_0^{'}], \, \delta_0^{''}:=[\Delta_0^{''}]$, and $\delta_0^{\mathrm{ram}}:=[\Delta_0^{\mathrm{ram}}]$ are boundary divisor classes on $\rr_g$ whose meaning will be recalled in Section 3. Note that up to a $\mathbb Z_2$-factor, a general point of $\Delta_0^{'}$ corresponds to a $2$-pointed Prym curve of genus $7$, for which we apply our Theorem \ref{unir7}. We establish the following result:

\begin{theorem}\label{r8}
The moduli space $\rr_8$ is uniruled.
\end{theorem}

Using the parametrization of $\cR_7$ via Nikulin surfaces, we construct a sweeping curve $\Gamma$ of the boundary divisor $\Delta_0^{'}$ of $\rr_8$ such that $\Gamma \cdot \delta_0^{'}>0$ and $\Gamma \cdot K_{\rr_8}<0$. This implies that the canonical class $K_{\rr_8}$ cannot be pseudoeffective, hence via \cite{BDPP}, the moduli space $\rr_8$ is uniruled. This way of showing uniruledness of a moduli space, though quite effective, does not lead to an \emph{explicit} uniruled parametrization of $\cR_8$. In Section 3, we sketch an alternative, more geometric way of showing that $\cR_8$ is uniruled, by embedding a general Prym-curve of genus $8$ in a certain canonical surface. A rational curve through a general point of $\rr_8$ is then induced by a pencil on this surface.

\section{Polarized Nikulin surfaces}

We briefly recall some basics on Nikulin surfaces, while referring to \cite{vGS}, \cite{GS} and \cite{Mo} for details. A \emph{symplectic involution} $\iota$ on a smooth $K3$ surface $Y$ has $8$ fixed points and we denote by $\bar{Y}:=Y/\langle \iota \rangle$ the quotient. The surface $\bar{Y}$ has $8$ nodes. Letting $\sigma:\widetilde{S}\rightarrow Y$ be the blow-up of the fixed points, the involution $\iota$ lifts to an involution $\tilde{\iota}:\widetilde{S}\rightarrow \widetilde{S}$ fixing the eight $(-1)$-curves $E_1, \ldots, E_8\subset \widetilde{S}$. Denoting by $f:\widetilde{S}\rightarrow S$ the quotient map by the involution $\widetilde{\iota}$, we obtain a smooth $K3$ surface $S$, together with a primitive embedding of the Nikulin lattice $\mathfrak{N}\cong E_8(-2)\hookrightarrow \mbox{Pic}(S)$, where $N_i=f(E_i)$ for $i=1, \ldots, 8$. In particular, the sum of rational curves $N:=N_1+\cdots+N_8$ is an even divisor on $S$, that is, there exists a class $e\in \mbox{Pic}(S)$ such that $e^{\otimes 2}=\OO_S(N_1+\cdots+N_8)$. The cover $f:\widetilde{S}\rightarrow S$ is branched precisely along the curves $N_1, \ldots, N_8$. The following diagram summarizes the notation introduced so far and will be used throughout the paper:
\begin{equation}\label{diagram}
\begin{CD}
{\widetilde{S}} @>{\sigma}>> {Y} \\
@V{f}VV @V{}VV \\
{S} @>{}>> {\bar{Y}} \\
\end{CD}
\end{equation}

Nikulin \cite{Ni} p.262  showed that the possible configurations of even sets of disjoint $(-2)$-curves on a $K3$ surface $S$ are only those consisting of either $8$ curves (in which case $S$ is a Nikulin surface as defined in this paper), or of $16$ curves, in which case $S$ is a Kummer surface. From this point of view, Nikulin surfaces appear naturally as the \emph{Prym analogues} of $K3$ surfaces.

\begin{definition} A \emph{polarized Nikulin surface} of genus $g$ consists of a smooth $K3$ surface and a primitive embedding $j$ of the lattice $\Lambda_g=\mathbb Z\cdot \mathfrak{c}\oplus \mathfrak N\hookrightarrow \mbox{Pic}(S)$, such that $\mathfrak{c}^2=2g-2$ and the class $j(\mathfrak{c})$ is nef.
\end{definition}

Polarized Nikulin surfaces of genus $g$ form an irreducible $11$-dimensional moduli space $\mathcal{F}_g^{\mathfrak{N}}$, see for instance \cite{Do1}. Structure theorems for $\F_g^{\mathfrak{N}}$ for genus $g\leq 6$ have been established in \cite{FV}. For instance the following result is proven in \emph{loc.cit.} for Nikulin surfaces of genus $g=6$. Let $V=\mathbb C^5$ and fix a smooth quadric $Q\subset \PP(V)$. Then one has a birational isomorphism, which, in particular, shows that $\F_6^{\mathfrak{N}}$ is unirational:
$$\F_6^{\mathfrak{N}}\stackrel{\cong}\dashrightarrow G\Bigl(7, \bigwedge^2 V \Bigr)^{\mathrm{ss}}\dblq \mathrm{Aut}(Q).$$

On the other hand, fundamental facts about $\F_g^{\mathfrak{N}}$ are still not known. For instance, it is not clear whether $\mathcal{F}_g^{\mathfrak{N}}$ is a variety of general type for large $g$. Nikulin surfaces have been recently used decisively in \cite{FK} to prove the Prym-Green Conjecture on syzygies of general Prym-canonical curves of even genus.

\vskip 3pt

For a polarized Nikulin surface $(S, j)$ of genus $g$ as above, we set $C:=j(\mathfrak{c})$ and then $H\equiv C-e\in \mbox{Pic}(S)$. It is shown in \cite{GS}, that for any Nikulin surface $S$ having minimal Picard lattice $\mbox{Pic}(S)=\Lambda_g$, the linear system $\OO_S(H)$ is very ample for $g\geq 6$. We compute that $H^2=2g-6$ and denote by $\phi_H:S\rightarrow \PP^{g-2}$ the corresponding embedding. Since $N_i\cdot H=1$ for $i=1, \ldots,8$, it follows that the images $\phi_H(N_i)\subset \PP^{g-2}$ are lines. The existence of two closely linked distinguished polarizations $\OO_S(C)$ and $\OO_S(H)$ of genus $g$ and $g-2$ respectively on any Nikulin surface is one of the main sources for the rich geometry of the moduli space $\F_g^{\mathfrak{N}}$ for $g\leq 6$, see \cite{FV} and \cite{vGS}.

\vskip 4pt

Suppose that $\bigl[S, j:\Lambda_7\hookrightarrow \mbox{Pic}(S)\bigr]$ is a polarized Nikulin surface of genus $7$. In this case $$\phi_H:S\hookrightarrow \PP^5$$ is a surface of degree $8$ which is a complete intersection of three quadrics. For each smooth curve $C\in |\OO_S(j(\mathfrak{c}))|$, we have that $[C, \eta:=e_C]\in \cR_7$. Since $\OO_C(1)=K_C\otimes \eta$, it follows that the restriction $\phi_{H|C}:C\hookrightarrow \PP^5$ is a Prym-canonically embedded curve of genus $7$. This assignment gives rise to the map $\chi_7:\mathcal{P}_7^{\mathfrak{N}}\rightarrow \cR_7$.

\vskip 3pt

Conversely, to a general Prym curve $[C, \eta]\in \cR_7$ we associate a unique Nikulin surface of genus $7$ as follows. We consider the Prym-canonical embedding $\phi_{K_C\otimes \eta}:C \hookrightarrow \PP^5$ and observe that $S:=\mbox{bs}(|\I_{C/\PP^5}(2)|)$ is a complete intersection of three quadrics, that is, if smooth, a $K3$ surfaces of degree $8$. In fact, $S$ is smooth for a general choice of $[C, \eta]\in \cR_7$, see \cite{FV} Proposition 2.3. We then set $N\equiv 2(C-H)\in \mbox{Pic}(S)$ and note that $N^2=-16$ and $N\cdot H=8$. Using the cohomology exact sequence
$$0\longrightarrow H^0(S, \OO_S(N-C))\longrightarrow H^0(S, \OO_S(N)) \longrightarrow H^0(C, \OO_C(N)) \longrightarrow 0,$$
since $\OO_C(N)$ is trivial, we conclude that the divisor $N$ is effective on $S$. It is shown in \emph{loc.cit.} that for a general $[C, \eta]\in \cR_7$, we have a splitting
$N=N_1+\cdots+N_8$ into a sum of $8$ disjoint lines with $C\cdot N_i=0$ for $i=1, \ldots, 8$. This turns $S$ into a Nikulin surface and explains the birational isomorphisms
$$\chi_7^{-1}:\mathcal{\P}_7^{\mathfrak N}\stackrel{\cong}\dashrightarrow \cR_7$$
referred to in the Introduction.

\vskip 3pt

Suppose now that $\bigl[S, \OO_S(C), N_8\bigr]\in \widehat{\F}_7^{\mathfrak{N}}$, that is, we single out a $(-2)$-curve in the Nikulin lattice. Writing $e^{\otimes 2}=\OO_C(N_1+\cdots+N_8)$, the choice of $N_8$ also determines the sum of the seven remaining lines $N_1+\cdots+N_7$, where $H\cdot N_i=1$, for $i=1, \ldots, 8$.  We compute
$$(C-N_1-\cdots-N_7)^2=-2 \ \  \mbox{ and } \ \ (C-N_1-\cdots-N_7)\cdot H=5,$$ in particular, there exists an effective divisor $R$ on $S$, with $R\equiv C-N_1-\cdots-N_7$. Note also that $R\cdot N_i=2$, for $i=1, \ldots, 7$, that is, $R\subset \PP^5$ comes endowed with seven bisecant lines.

\begin{proposition}\label{vanish} For a decorated Nikulin surface $\bigl[S, \OO_S(C), N_8\bigr]\in \widehat{\F}_7^{\mathfrak{N}}$ satisfying $\mathrm{Pic}(S)=\Lambda_7$, we have that
$H^1(S, \OO_S(C-N_1-\cdots-N_7))=0$. In particular, $$R\in |\OO_S(C-N_1-\cdots-N_7)|$$ is a smooth rational quintic curve on $S$.
\end{proposition}

\begin{proof}
Assume by contradiction that the curve $R\subset S$ is reducible. In that case, there exists a smooth irreducible $(-2)$-curve $Y\subset S$, such that $Y\cdot R<0$ and
$H^0(S, \OO_S(R-Y))\neq 0$. Assuming $\mbox{Pic}(S)$ is generated by $C$, $N_1, \ldots, N_8$ and the class $e=(N_1+\cdots+N_8)/2$, there exist integers $a, b, c_1, \ldots, c_8\in \mathbb Z$, such that
$$Y\equiv a\cdot C+\Bigl(c_1+\frac{b}{2}\Bigr)\cdot N_1+\cdots+\Bigl(c_8+\frac{b}{2}\Bigr)\cdot N_8.$$
Setting $b_i:=c_i+\frac{b}{2}$, the numerical hypotheses on $Y$ can be rewritten in the following form:
\begin{equation}\label{ineq}
b_1^2+\cdots+b_8^2=6a^2+1  \ \mbox{ and } 6a+b_1+\cdots+b_8\leq -1.
\end{equation}
Since $Y$ is effective, we find that $a\geq 0$ (use that $C\subset S$ is nef). Applying the same considerations to the effective divisor $R-Y$,
we obtain that $a\in \{0,1\}$.

\vskip 3pt

If $a=0$, then $Y\equiv b_1 N_1+\cdots+b_8 N_8\geq 0$, hence $b_i\geq 0$ for $i=1, \ldots, 8$, which contradicts the inequality
$b_1+\cdots+b_8\leq -1$, so this case does not appear.

\vskip 2pt

If $a=1$, then $R-Y\equiv -(1+b_1) N_1-\cdots -(1+b_7) N_7-b_8 N_8\geq 0$, therefore
$b_8\leq 0$ and $b_i\leq -1$ for $i=1, \ldots,7$. From (\ref{ineq}), we obtain that $b_8=0$ and $b_1=\cdots=b_7=-1$. Thus $Y\equiv R$, which is a contradiction, for $Y$ was assumed to be a proper irreducible component of $R$.
\end{proof}

Retaining the notation above, we obtain a  map $\psi:\widehat{\F}_7^{\mathfrak{N}}\dashrightarrow \mathfrak{Rat}_7$,  defined by
$$\psi\Bigl([S, \OO_S(C), N_8]\Bigr):=[R,  \ N_1\cdot R+\cdots+N_7\cdot R],$$
where the cycle $N_i\cdot R\in \mbox{Sym}^2(R)$ is regarded as an effective divisor of degree $2$ on $R$.  The map $\psi$ is regular over the dense open subset of $\widehat{\F}_7^{\mathfrak{N}}$ consisting of Nikulin surfaces having the minimal Picard lattice $\Lambda_7$. We are going to show that $\psi$ is a birational isomorphism by explicitly constructing its inverse. This will be the map $\varphi$ described in the Introduction in Theorem \ref{m14}.

\vskip 5pt

We fix a smooth rational quintic curve $R\subset \PP^5$ and recall the canonical identification
\begin{equation}\label{ident}
\bigl|\I_{R/\PP^5}(2)\bigr|=\bigl|\OO_{\mathrm{Sym}^2(R)}(3)\bigr|
\end{equation}
between the linear system of quadrics containing $R\subset \PP^5$ and that of plane cubics.
Here we use the isomorphism
$\mbox{Sym}^2(R)\stackrel{\cong}\longrightarrow \PP^2,$  under which to a quadric $Q\in H^0(\PP^5, \I_{R/\PP^5}(2))$ one assigns the symmetric correspondence
$$\Sigma_Q:=\{x+y\in \mbox{Sym}^2(R): \langle x,y\rangle \subset Q\},$$
 which is a cubic curve in $\mbox{Sym}^2(R)$.

Let $N_1, \ldots, N_7$ be general bisecant lines to $R$ and consider the nodal curve of genus $7$
$$C:=R\cup N_1\cup \ldots \cup N_7\subset \PP^5.$$

\begin{proposition}
For a general choice of the bisecants $N_1, \ldots, N_7$ of the curve $R\subset \PP^5$, the base locus
$$S:=\mathrm{Bs}\ \bigl|\I_{C/\PP^5}(2)\bigr|$$
is a smooth $K3$ surface of degree $8$.
\end{proposition}
\begin{proof}
The bisecant line $N_i$ is determined by the degree $2$ divisor $N_i\cdot R\in \mbox{Sym}^2(R)$. Under the identification (\ref{ident}), the quadrics containing the line $N_i$ are identified with the cubics in $\bigl|\OO_{\mathrm{Sym}^2(R)}(3)\bigr|$ that pass through the point $N_i\cdot R$. It follows that the linear system  $\bigl |\I_{C/\PP^5}(2)\bigr|$ corresponds to the linear system of cubics in $\mbox{Sym}^2(R)$ passing through $7$ general points. Since the secants $N_i$ (and hence the points $N_i\cdot R\in \mbox{Sym}^2(R)$) have been chosen to be general, we obtain that $\mbox{dim } |\I_{C/\PP^5}(2)|=2$.

We have proved in Proposition \ref{vanish} that for a general Nikulin surface $[S, \OO_S(C)]\in \mathcal{F}_7^{\mathfrak{N}}$ we have
$$H^1(S, \OO_S(C-N_1-\cdots-\hat{N}_i-\cdots-N_8))=0,$$ and the corresponding curves $R_i\in \bigl|\OO_S(C-N_1-\cdots-\hat{N}_i-\cdots-N_8)\bigr|$ are smooth rational quintics for $i=1, \ldots, 8$. In particular, the morphism $\psi:\widehat{\F}_7^{\mathfrak{N}}\dashrightarrow \mathfrak{Rat}_7$ is defined on all components of $\widehat{\F}_7^{\mathfrak{N}}$ and the image of each component is an element of $\mathfrak{Rat}_7$ (a priori, one does not know that $\widehat{\F}_7^{\mathfrak{N}}$ is irreducible, this will follow from our proof).  For such a point in $\mbox{Im}(\psi)$, it follows that the base locus $\mbox{bs } \bigl|\I_{C/\PP^5}(2)\bigr|$ is a smooth surface, in fact a general Nikulin surface of genus $7$. Hence $[S,\OO_S(C), N_i]\in \mbox{Im}(\varphi)$ for $i=1, \ldots, 8$. Since $\mathfrak{Rat}_7$ is an irreducible variety, the conclusion follows.
\end{proof}

\noindent \emph{Proof of Theorem \ref{m14}}. As explained in the Introduction, the map $\varphi:\mathfrak{Rat}_7\dashrightarrow \widehat{\F}_7^{\mathfrak{N}}$ is well-defined and clearly the inverse of $\psi$. In particular, it follows that $\widehat{\F}_7^{\mathfrak{N}}$ is also irreducible (and in fact unirational).
\hfill $\Box$

\section{Configuration spaces of points in the plane}

Throughout this section we use the identification $\mbox{Sym}^2(\PP^1)\cong \PP^2$ induced by the map $\rho:\PP^1\times \PP^1\rightarrow \PP^2$ obtained by taking the projection of the Segre embedding of $\PP^1\times \PP^1$ to the space of symmetric tensors, that is, $\rho\bigl([a_0,a_1],[b_0,b_1]\bigr)=[a_0 b_0, a_1b_1, a_0b_1+a_1b_0]$. We identify the diagonal $\Delta\subset \PP^1\times \PP^1$ with its image $\rho(\Delta)$ in $\PP^2$. We view $PGL(2)$ as the subgroup of automorphisms of $\PP^2$ that preserve the conic $\Delta$. Furthermore, the choice of $\Delta$ induces a canonical identification
$$PGL(3)/PGL(2)=|\OO_{\PP^2}(2)|=\PP^5.$$

For $g\geq 5$, we consider the projection
$$\beta:\mathfrak{Rat}_g:=\mbox{Hilb}^g(\PP^2)\dblq SL(2)\rightarrow \mbox{Hilb}^g(\PP^2)\dblq SL(3)=:U_g^2.$$

\begin{definition}
If $X$ is a del Pezzo surface of degree $2$, a \emph{contraction} of $X$ is the blow-up $f:X\rightarrow \PP^2$ of $7$ points in general position in $\PP^2$.
\end{definition}

Specifying a pair $(X,f)$ as above, amounts to giving a \emph{plane model} of the del Pezzo surface, that is, a pair $(X,L)$, where $X$ is a del Pezzo surface with $K_X^2=2$ and $L\in \mbox{Pic}(S)$ is such that $L^2=1$ and $K_X\cdot L=-2$. Therefore $U_7^2$ is the GIT moduli space of pairs $(X,f)$ (or equivalently of pairs $(X,L)$) as above.

\begin{proposition}\label{descent}
The morphism $\beta: \mathrm{Hilb}^g(\PP^2)\dblq SL(2)\rightarrow U_g^2$ is a locally trivial $\PP^5$-fibration.
\end{proposition}
\begin{proof}
Having fixed the conic $\Delta\subset \PP^2$, we have an identification $\PP^2\cong  \mbox{Sym}^2(\Delta)\cong (\PP^2)^{\vee}$, that is, we view points in $\mbox{Sym}^2(\Delta)$ as lines in $\PP^2$. A general point $D\in \mbox{Hilb}^g(\PP^2)$ corresponds to a union $D=\ell_1+\cdots + \ell_g$ of $g$ lines in $\PP^2$, such that $\mbox{Aut}(\{\ell_1, \ldots, \ell_g\})=1$.
We consider the rank $6$ vector bundle $\E$ over $\mbox{Hilb}^g(\PP^2)$ with fibre
$$\E(\ell_1+\cdots+\ell_g):=H^0\bigl(\OO_{\ell_1+\cdots+\ell_g}(2)\bigr).$$
Clearly $\E$ descends to a vector bundle $E$ over the quotient $U^2_g$. We then observe that one has a canonical identification $\PP(E)\cong \mbox{Hilb}^g(\PP^2)\dblq SL(2)$, or more geometrically, $\ra_g$ is the moduli space of pairs consisting of an unordered configuration of $g$ lines and a conic in $\PP^2$. The birational isomorphism
$\PP(E)\rightarrow \mbox{Hilb}^g(\PP^2)\dblq SL(2) $ is given by the assignment
$$\Bigl(\ell_1+\cdots+\ell_g, Q\Bigr) \mbox{ mod } SL(3)\mapsto \sigma(\ell_1)+\cdots+\sigma(\ell_g) \mbox{ mod } SL(2),$$
where $\sigma \in SL(3)$ is an automorphism such that $\sigma(Q)=\Delta$.
\end{proof}

\vskip 3pt

\noindent \emph{Proof of Theorem \ref{rat8}}. We have established that the moduli space $\widehat{\F}_7^{\mathfrak{N}}$ is birationally isomorphic to the projectivization of a $\PP^5$-bundle
over $U_7^2$. Since $U_7^2$ is rational, cf. \cite{Bo} Theorem 2.2.4.2, we conclude.
\hfill $\Box$

\begin{remark} In view of Theorem \ref{rat8}, it is natural to ask whether there exists a rational \emph{modular} degree $8$ cover $\widehat{\cR}_7\rightarrow \cR_7$ which is a locally trivial $\PP^7$-bundle over the rational variety $\widehat{\F}_7^{\mathfrak{N}}$, such that the following diagram is commutative:
$$
     \xymatrix{
         \widehat{\cR}_{7} \ar[r]^{?} \ar[d]_{8:1} & \widehat{\F}_{7}^{\mathfrak{N}} \ar[d]_{8:1} \ar[r]^{\cong} & \mathfrak{Rat}_7   \\
          \cR_7   \ar[r]^{\PP^7}     & \F_7^{\mathfrak{N}} \\}
$$
One candidate for the cover $\widehat{\cR}_7$ is the universal singular locus of the Prym-theta divisor,
$$\widehat{\cR}_7:=\Bigl\{[C, \eta, L]\in \cR_7: [C, \eta]\in \cR_7 \mbox{ and } L\in \mathrm{Sing}(\Xi)/\pm\Bigr\},$$
where $\mbox{Sing}(\Xi)=\{L\in \mbox{Pic}^{2g-2}(\widetilde{C}):\mathrm{Nm}_f(L)=K_C, h^0(C,L)\geq 4, h^0(C,L)\equiv 0 \mbox{ mod } 2\}.$
It is shown in \cite{De} that for a general point $[C, \eta]\in \cR_7$, the locus $\mbox{Sing}(\Xi)$ is reduced and consists of $16$ points, so indeed
$\mbox{deg}(\widehat{\cR}_7/\cR_7)=8$. So far we have been unable to construct the required map $\widehat{\cR}_7\rightarrow \widehat{\F}_7^{\mathfrak{N}}$
and we leave this as an open question.
\end{remark}

\section{The uniruledness of $\rr_8$}

We now explain how our structure results on $\F_7^{\mathfrak{N}}$ and $\cR_7$ lead to an easy proof of the uniruledness of $\rr_8$. We begin by reviewing  a few facts about the compactification $\rr_g$ of $\cR_g$ by means of stable Prym curves, see \cite{FL} for details.  The geometric points of the coarse moduli space $\rr_g$ are triples $(X, \eta, \beta)$, where $X$ is a quasi-stable curve of genus $g$, $\eta\in \mbox{Pic}(X)$ is a line bundle of total degree is $0$ such that $\eta_{E}=\OO_E(1)$ for each smooth rational component $E\subset X$ with $|E\cap \overline{X-E}|=2$  (such a component is said to be \emph{exceptional}), and $\beta:\eta^{\otimes 2}\rightarrow \OO_X$ is a sheaf homomorphism whose restriction to any non-exceptional component is an isomorphism. If $\pi:\rr_g\rightarrow \mm_g$ is the map dropping the Prym structure, one has the formula \cite{FL}
\begin{equation}\label{pullbackrg}
\pi^*(\delta_0)=\delta_0^{'}+\delta_0^{''}+2\delta_{0}^{\mathrm{ram}}\in CH^1(\rr_g),
\end{equation}
where $\delta_0^{'}:=[\Delta_0^{'}], \, \delta_0^{''}:=[\Delta_0^{''}]$, and $\delta_0^{\mathrm{ram}}:=[\Delta_0^{\mathrm{ram}}]$ are irreducible boundary divisor classes on $\rr_g$, which we describe by specifying their respective general points.

\vskip 3pt

We choose a general point $[C_{xy}]\in \Delta_0\subset \mm_g$ corresponding to a smooth $2$-pointed curve $(C, x, y)$ of genus $g-1$ and consider the normalization map $\nu:C\rightarrow C_{xy}$, where $\nu(x)=\nu(y)$. A general point of $\Delta_0^{'}$ (respectively of $\Delta_0^{''}$) corresponds to a pair $[C_{xy}, \eta]$, where $\eta\in \mbox{Pic}^0(C_{xy})[2]$ and $\nu^*(\eta)\in \mbox{Pic}^0(C)$ is non-trivial
(respectively, $\nu^*(\eta)=\OO_C$). A general point of $\Delta_{0}^{\mathrm{ram}}$ is a Prym curve of the form $(X, \eta)$, where $X:=C\cup_{\{x, y\}} \PP^1$ is a quasi-stable curve with $p_a(X)=g$ and  $\eta\in \mbox{Pic}^0(X)$ is a line bundle such that $\eta_{\PP^1}=\OO_{\PP^1}(1)$ and $\eta_C^{\otimes 2}=\OO_C(-x-y)$. In this case, the choice of the homomorphism $\beta$ is uniquely determined by $X$ and $\eta$. Therefore, we drop $\beta$ from the notation of such a Prym curve.

There are similar decompositions of the pull-backs $\pi^*([\Delta_j])$ of the other boundary divisors $\Delta_j\subset \mm_g$ for $1\leq j\leq \lfloor \frac{g}{2}\rfloor$, see again \cite{FL} Section 1 for details.

\vskip 4pt

Via Nikulin surfaces we construct a sweeping curve for the divisor $\Delta^{'}_0\subset \rr_8$. Let us start with a general element of $\Delta_0^{'}$ corresponding to a smooth $2$-pointed curve $[C,x,y]\in \cM_{7,2}$ and a $2$-torsion point $\eta\in \mbox{Pic}^0(C_{xy})[2]$ and set $\eta_C:=\nu^*(\eta)\in \mbox{Pic}^0(C)[2]$. Using \cite{FV} Theorem 0.2, there exists a Nikulin surface $f:\widetilde{S}\rightarrow S$ branched along $8$ rational curves $N_1, \ldots, N_8\subset S$ and an embedding $C\subset S$, such that $C\cdot N_i=0$ for $i=1, \ldots, 8$ and $\eta_C=e_C$, where $e\in \mbox{Pic}(S)$ is the even class with $e^{\otimes 2}=\OO_S(N_1+\cdots+N_8)$. We can also assume that $\mbox{Pic}(S)=\Lambda_7$. By moving $C$ in its linear system on $S$, we may assume that $x,y\notin N_1\cup \ldots \cup N_8$, and we set $\{x_1, x_2\}=f^{-1}(x)$ and $\{y_1,y_2\}=f^{-1}(y)$.

\vskip 5pt

We pick a Lefschetz pencil  $\Lambda:=\{C_t\}_{t\in \PP^1}$ consisting of curves on $S$ passing through the points $x$ and $y$. Since the locus $\bigl\{D\in |\OO_S(C)|: D\supset N_i\bigr\}$ is a hyperplane in $|\OO_S(C)|$, it follows that there are precisely eight distinct values $t_1, \ldots, t_8\in \PP^1$ such that
$$C_{t_i}=:C_i=N_i+D_i,$$
where $D_i$ is a smooth curve of genus $6$ which contains $x$ and $y$ and intersects $N_i$ transversally at two points.  For each $t\in \PP^1-\{t_1, \ldots, t_8\}$, we may assume that $C_t$ is a smooth curve and denoting $[\bar{C}_t:=C_t/x\sim y]\in \mm_8$, we have an exact sequence
$$0\longrightarrow \mathbb Z_2 \longrightarrow \mbox{Pic}^0(\bar{C}_t)[2]\longrightarrow \mbox{Pic}^0(C_t)[2]\longrightarrow 0.$$
In particular, there exist two distinct line bundles $\eta_t^{'}, \eta_t^{''}\in \mbox{Pic}^0(\bar{C}_t)$  such that
$$\nu_t^*(\eta_t^{'})=\nu_t^*(\eta_t^{''})=e_{C_t}.$$
Using the Nikulin surfaces, we can consistently distinguish $\eta_t^{'}$ from $\eta_t^{''}$. Precisely, $\eta_t^{'}$ corresponds to the admissible cover
$$f^{-1}(C_t)/x_1\sim y_1, x_2\sim y_2 \stackrel{2:1}\longrightarrow \bar{C}_t$$
whereas $\eta_t^{''}$ corresponds to the admissible cover
$$f^{-1}(C_t)/x_1\sim y_2, x_2\sim y_1 \stackrel{2:1}\longrightarrow \bar{C}_t.$$

\vskip 5pt

First we construct the pencil $R:=\{\bar{C}_t\}_{t\in \PP^1}\hookrightarrow \mm_8$. Formally, we have a fibration $u:\mbox{Bl}_{2g-2}(S)\rightarrow \PP^1$ induced by the pencil $\Lambda$ by blowing-up $S$ at its $2g-2$ base points (two of which being $x$ and $y$ respectively), which comes endowed with sections $E_x$ and $E_y$ given by the corresponding exceptional divisors.
The pencil $R$ is obtained from $u$, by identifying inside the surface $\mbox{Bl}_{2g-2}(S)$ the sections $E_x$ and $E_y$ respectively.

\begin{lemma}\label{int11}
The pencil $R\subset \mm_8$ has the following numerical characters:
$$R\cdot \lambda=g+1=8, \ \ R\cdot \delta_0=6g+16=58, \ \mbox{ and } \  R\cdot \delta_j=0 \ \  \mbox{ for } j=1, \ldots, 4.$$
\end{lemma}
\begin{proof}
We observe that $(R\cdot \lambda)_{\mm_8}=(\Lambda \cdot \lambda)_{\mm_7}=g+1=8$ and $(R\cdot \delta_j)_{\mm_8}=(\Lambda \cdot \delta_j)_{\mm_7}=0$ for $j\geq 1$. Finally, in order to determine the degree of the normal bundle of $\Delta_0$ along $R$, we write:
$$(R\cdot \delta_0)_{\mm_8}=(\Lambda  \cdot \delta_0)_{\mm_7}+E_x^2+E_y^2=6g+18-2=58,$$
where we have used the well-known fact that a Lefschetz pencil of curves of genus $g$ on a $K3$ surface possesses $6g+18$ singular fibres (counted with their multiplicities) and that $E_x^2=E_y^2=-1$.
\end{proof}

\vskip 4pt

Next, note that the family of Prym curve 
$\Bigl\{[\bar{C_t}, \eta_t \bigr]: \nu_t^*(\eta_t)=e_{C_t}\Bigr\}_{t\in \PP^1}\hookrightarrow \rr_8$ splits into two irreducible components meeting in eight points. We consider
one of the irreducible components, say
$$\Gamma:=\Bigl\{[\bar{C}_t,\eta_t^{'}]\Bigr\}_{t\in \PP^1}\hookrightarrow \rr_8,$$
where the notation for $\eta_t^{'}$ has been explained above.

\begin{lemma}\label{int12}
The curve $\Gamma\subset \rr_8$ constructed above has the following numerical features:
$$\Gamma \cdot \lambda=8, \ \ \Gamma\cdot \delta_0^{'}=42, \ \ \Gamma\cdot \delta_0^{''}=0 \mbox{ and } \Gamma\cdot \delta_0^{\mathrm{ram}}=8.$$
Furthermore, $\Gamma$ is disjoint from all boundary components contained in $\pi^*(\Delta_j)$ for $j=1, \ldots, 4$.
\end{lemma}
\begin{proof} First we observe that $\Gamma$ intersects the divisor $\Delta_0^{\mathrm{ram}}$ transversally at the points corresponding to the values $t_1, \ldots, t_8\in \PP^1$, when the curve $C_{i}$ acquires the $(-2)$-curve $N_i$ as a component. Indeed, for each of these points $e^{\otimes (-2)}_{D_i}=\OO_{D_i}(-N_i)$ and  $e^{\vee}_{N_i}=\OO_{N_i}(1)$, therefore $[C_i, e_{C_i}]\in \Delta_0^{\mathrm{ram}}$. Furthermore, using Lemma \ref{int11} we write $(\Gamma\cdot \lambda)_{\rr_8}=\pi_*(\Gamma)\cdot \lambda=8$ and
$$\Gamma\cdot (\delta_0^{'}+\delta_0^{''}+2\delta_0^{\mathrm{ram}})=\Gamma \cdot \pi^*(\delta_0)=R\cdot \delta_0=58.$$
Furthermore, for $t\in \PP^1-\{t_1, \ldots, t_8\}$, the curve $f^{-1}(C_t)$ cannot split into two components, else $\mbox{Pic}(S)\varsupsetneq \Lambda_7$. Therefore $\gamma\cdot \delta_0^{''}=0$ and hence $\Gamma\cdot \delta_0^{'}=42$.
\end{proof}

\vskip 4pt

\noindent \emph{Proof of Theorem \ref{r8}.} The curve $\Gamma\subset \rr_8$ constructed above is a sweeping curve for the irreducible boundary divisor $\Delta_0^{'}$, in particular it intersects non-negatively every irreducible effective divisor $D$ on $\rr_8$ which is different from $\Delta_0^{'}$. Since $\Gamma \cdot \delta_0^{'}>0$, it follows that $D$ intersects non-negatively \emph{every} pseudoeffective divisor on $\rr_8$. Using  the formula for the canonical divisor \cite{FL}
$$K_{\rr_8}=13\lambda-2(\delta_0^{'}+ \delta_0^{''})-3\delta_0^{\mathrm{ram}}-\cdots \in CH^1(\rr_8),$$ applying  Lemma \ref{int12} we obtain that $\Gamma \cdot K_{\rr_8}=-4<0$, thus $K_{\rr_8}\notin \mbox{Eff}(\rr_8)$. Using \cite{BDPP}, we conclude that $\rr_8$ is uniruled, in particular its Kodaira dimension is negative.
\hfill $\Box$

\vskip 5pt

\subsection{The uniruledness of the universal singular locus of the theta divisor over $\rr_8$.}\hfill

\vskip 4pt

In what follows, we sketch a second proof of Theorem \ref{r8}, skipping some details. This parametrization provides a \emph{concrete} way of constructing a rational curve through a general point of $\rr_8$. We fix a general element $[C, \eta]\in \cR_8$ and denote by $f:\widetilde{C}\rightarrow C$ the corresponding unramified double cover and by $\iota:\widetilde{C}\rightarrow \widetilde{C}$ the involution exchanging the sheets of $f$. Following \cite{W}, we consider the singular locus of the Prym theta divisor, that is,  the locus
$$V^3(C,\eta) =\mbox{Sing}(\Xi):= \bigl\{L\in \mathrm{Pic}^{14}(\widetilde{C}): \mathrm{Nm}_f(L) = K_C, \ h^0(C, L)\geq 4 \hbox{ and } h^0(C, L)\equiv 0 \ \mathrm{mod}\ 2\bigr\}.$$
It follows from \cite{W}, that $V^3(C, \eta)$ is a smooth curve. We pick a line bundle $L\in V^3(C, \eta)$ with $h^0(\widetilde{C}, L)=4$, a general point $\tilde{x}\in \widetilde{C}$ and consider the  $\iota$-invariant part of the Petri map, that is,
$$\mu_0^+\bigl(L(-\tilde{x})\bigr): \mathrm{Sym}^2 H^0(\widetilde{C}, L(-\tilde{x})) \rightarrow H^0(C, K_C(-x)),$$
$$\ \ \mbox{  }  s\otimes t+t\otimes s\mapsto s\cdot \iota^*(t)+t\cdot \iota^*(s),$$
where $x:=f(\tilde{x})\in C$.
We set $\PP^2:=\PP\bigl(H^0(L(-\tilde{x}))^{\vee}\bigr)$, and similarly to \cite{FV} Section 2.2, we consider the map $q:\PP^2\times \PP^2\rightarrow \PP^5$ obtained from the Segre embedding $\PP^2\times \PP^2\hookrightarrow \PP^8$ by projecting onto the space of symmetric tensors. We have the following commutative diagram:
\begin{figure}[h]
$$\xymatrix@R=11pt{\widetilde{C} \ar[rr]^-{\bigl(L(-\tilde{x}), \iota^{*}(L(-\tilde{x}))\bigr)} \ar[dd]_f && \PP^2\times \PP^2 \ar[dd]_q \ar@{^{}->}[rd] \\
&& & \PP^8=\PP\bigl(H^0 (L(-\tilde{x}))^{\vee}\otimes H^0(L(-\tilde{x}))^{\vee}\bigr) \ar@{-->}@<-1ex>[dl] \\
C\ar[rr]^-{\bigl|\mu_0^+(L(-\tilde{x}))\bigr|} && *!<-22pt,0pt>{\PP^5=\PP(\mathrm{Sym}^2 H^0\bigl(L(-\tilde{x})\bigr)^{\vee})}}
$$
\end{figure}

Let $\Sigma:=\mbox{Im}(q)\subset \PP^5$ be the determinantal cubic surface; its singular locus is the Veronese surface $V_4$. For a general choice of $[C, \eta]\in \cR_8, L\in V^3(C,\eta)$ and of $\tilde{x}\in \widetilde{C}$, the map
$\mu_0^+(L(-\tilde{x}))$ is injective and let $W\subset H^0(C, K_C(-x))$ be its $6$-dimensional image. Comparing dimensions, we observe that the kernel of the multiplication map
$$\mbox{Sym}^2(W)\longrightarrow H^0(C, K_C^{\otimes 2}(-2x))$$
is at least $2$-dimensional. In particular, there exist distinct quadrics $Q_1, Q_2\subset \PP^5$ such that
$$C\subset S:=Q_1\cap Q_2\cap \Sigma\subset \PP^5.$$
Since $\mbox{Sing}(\Sigma)=V_4$, the surface $S$ is singular at the $16$ points of intersection $Q_1\cap Q_2\cap V_4$, or equivalently, $\mbox{Sing}(S)\supseteq Q_1\cap Q_2\cap V_4$. 
Assume now, we can find $(C,L, \eta, \tilde{x})$ as above such that $S$ has no further singularities except the already exhibited $16$ points, that is, 
 $$\mbox{Sing}(S)=Q_1\cap Q_2\cap V_4.$$
We obtain that $S$ is a $16$-nodal canonical surface, that is, $K_S=\OO_S(1)$.

\vskip 4pt

Using the exact sequence $0\rightarrow H^0(S,\OO_S)\rightarrow H^0(S, \OO_S(C))\rightarrow H^0(\OO_C(C))\rightarrow 0$, since $\OO_C(C)=\OO_C(x)$, we find that
$\mbox{dim } |\OO_S(C)|=1$, that is, $C$ moves on $S$. Moreover the pencil $|\OO_S(C)|$ has $x\in S$ as a base point.

\vskip 3pt

We consider the surface $\widetilde{S}:=q^{-1}(S)\subset \PP^2\times \PP^2$. For each curve $C_t\in |\OO_S(C)|$, we denote by $\widetilde{C}_t:=q^{-1}(C_t)\subset \widetilde{S}$ the corresponding double cover. Furthermore, we define a line bundle $L_t\in \mbox{Pic}^{14}(\widetilde{C}_t)$, by setting $\OO_{\widetilde{C}_t}(1,0)=L_t(-\tilde{x})$ (in which case, $\OO_{\widetilde{C}_t}(0,1)=\iota^*(L_t(-\tilde{x}))$).

\vskip 4pt

The construction we just explained induces a uniruled parametrization of the universal singular locus of the Prym theta divisor in genus $8$ (which dominates $\cR_8$). Our result is conditional to a (very plausible) transversality assumption:

\begin{theorem}\label{r83}
Assume there exists $[C, \eta, L, x]$ as above,  such that $S=Q_1\cap Q_2\cap \Sigma\subset \PP^5$ is a $16$-nodal canonical surface. Then the moduli space
$$\cR_8^3:=\Bigl\{[C,\eta,L]: [C, \eta]\in \cR_8, \ L\in V^3(C,\eta)\Bigr\}$$ is uniruled.
\end{theorem}
\begin{proof}
The assignment $\PP^1\ni t\mapsto [\widetilde{C}_t/C_t, L_t]\in \cR_8^3$ described above provides a rational curve passing through a general point of $\cR_8^3$.
\end{proof}



\begin{thebibliography}{aaaaa}
\bibitem[BDPP]{BDPP} S. Boucksom, J.P. Demailly, M. P\u{a}un and T. Peternell, {\em{The pseudo-effective cone of a compact K\"ahler manifold and varieties of negative Kodaira dimension}},
Journal of Algebraic Geometry \textbf{22} (2013), 201-248.
\bibitem[Bo]{Bo} C. B\"ohning, {\em{The rationality problem in algebraic geometry}}, Habilitationsschrift, G\"ottingen 2009.
\bibitem[Br]{Br} G. Bruns, {\em{$\cR_{15}$ is of general type}}, arXiv:1511.08468.
\bibitem[Cas]{Cas} G. Castelnuovo, {\em{Numero delle involuzioni razionali giacenti sopra una curva di dato
  genere}},  Rendiconti R. Accad. Lincei \textbf{5} (1889) 130-133.
\bibitem[Cat]{Cat} F. Catanese, {\em{On the rationality of certain moduli spaces
related to curves of genus $4$}}, Springer Lecture Notes in Mathematics \textbf{1008} (1983), 30-50.
\bibitem[De]{De} O. Debarre, {\em{Sur le probleme de Torelli pour les varietes de Prym}}, American J. Math. \textbf{111} (1989), 111-134.
\bibitem[Do1]{Do1} I. Dolgachev, {\em{Mirror symmetry for lattice polarized $K3$ surfaces}}, J. Math. Sciences \textbf{81} (1996), 2599-2630.
\bibitem[Do2]{Do2} I. Dolgachev, {\em{Rationality of $\cR_2$ and $\cR_3$}}, Pure and Applied Math. Quarterly, \textbf{4} (2008), 501-508.
\bibitem[Do3]{Do3} I. Dolgachev, {\em{Classical algebraic geometry: A modern view}}, Cambridge University Press 2012.
\bibitem[Do]{Do} R. Donagi, {\em{The unirationality of $\mathcal{A}_5$}}, Annals of Mathematics \textbf{119} (1984), 269-307.
\bibitem[FK]{FK} G. Farkas and M. Kemeny, {\em{The generic Green-Lazarsfeld conjecture}}, arXiv:1408.4164, to appear in Inventiones Math.
\bibitem[FL]{FL} G. Farkas and K. Ludwig, {\em{The Kodaira dimension of the moduli
space of Prym varieties}}, Journal of the European Mathematical Society \textbf{12} (2010), 755-795.
\bibitem[FV]{FV} G. Farkas and A. Verra, {\em{Moduli of theta-characteristics via Nikulin surfaces}}, Mathematische Annalen \textbf{354} (2012), 465-496.
\bibitem[GS]{GS} A. Garbagnati and A. Sarti, {\em{Projective models of $K3$ surfaces with an even set}}, Advances in Geometry \textbf{8} (2008), 413-440.
\bibitem[GKM]{GKM} A. Gibney, S. Keel and I. Morrison, {\em{Towards the ample cone of $\mm_{g,n}$}}, Journal of the American Mathematical Society \textbf{15} (2002), 273-294.
\bibitem[vGS]{vGS} B. van Geemen and A. Sarti, {\em{Nikulin involutions on $K3$ surfaces}}, Mathematische Zeitschrift \textbf{255} (2007), 731-753.
\bibitem[ILS]{ILS} E. Izadi, M. Lo Giudice and G. Sankaran, {\em{The
moduli space of \'etale double covers of genus $5$ is unirational}}, Pacific Journal of Mathematics \textbf{239} (2009), 39-52.
\bibitem[Ka]{Ka} P. Katsylo, {\em{On the birational geometry of $(\PP^n)^{(m)}/GL_{n+1}$}}, Max-Planck Institut Preprint, MPI/94-144 (1994).
\bibitem[Mo]{Mo} D. Morrison, {\em{On $K3$ surfaces with large Picard number}}, Inventiones Math. \textbf{75} (1984), 105-121.
\bibitem[Mu]{Mu} S. Mukai, {\em{Curves and $K3$ surfaces of genus eleven}}, in: Moduli of vector bundles, Lecture Notes in Pure and Applied Mathematics Vol. 179, Dekker (1996), 189-197.
\bibitem[Ni]{Ni} V.V. Nikulin, {\em{Kummer surfaces}}, Izvestia Akad. Nauk SSSR \textbf{39} (1975), 278-293.
\bibitem[V]{V} A. Verra, {\emph{A short proof of the unirationality
of $\cA_5$}}, Indagationes Math. \textbf{46} (1984), 339-355.
\bibitem[W]{W} G. Welters,
\emph{A theorem of Gieseker-Petri type for Prym varieties},
Annales Sci. \'Ecole  Normale Sup\'erieure \textbf{18} (1985), 671--683.

\end{thebibliography}
\end{document}